\documentclass[reqno,12pt]{amsart}
\usepackage{amsfonts,amsmath,amsthm,amssymb,amscd,verbatim}
\usepackage[all]{xy}
\newcommand{\version}{Ver.~0.0}
\newcommand{\setversion}[1]{\renewcommand{\version}{Ver.~{#1}}}
\setversion{0.? [2012/04/02 19:10]}
\setversion{1.0 [2012/04/04 12:41]}

\title [Closed orbits and double flag variety]
{Closed orbits on partial flag varieties and double flag variety of finite type}

\author[Kondo]{Kensuke Kondo}
\address{
Department of Physics and Mathematics\\
Aoyama Gakuin University\\
Fuchinobe 5-10-1, Chuo-ku, Sagamihara 252-5258, Japan}
\email{kenkenpa826@hotmail.co.jp}

\author[Nishiyama]{Kyo Nishiyama$^{\ast}$}
\address{
Department of Physics and Mathematics\\
Aoyama Gakuin University\\
Fuchinobe 5-10-1, Chuo-ku, Sagamihara 252-5258, Japan}
\email{kyo@gem.aoyama.ac.jp}

\thanks{$\ast)${\ } Supported by JSPS Grant-in-Aid for Scientific Research (B) \#{21340006}.}

\author[Ochiai]{Hiroyuki Ochiai$^{\dagger}$}
\address{
Institute of Mathematics-for-Industry\\
Kyushu University\\
744, Motooka, Nishi-ku, Fukuoka 819-0395, Japan}
\email{ochiai@imi.kyushu-u.ac.jp}

\thanks{$\dagger)${\ } Supported by JSPS Grant-in-Aid for Scientific Research (A) \#{19204011}.}

\author[Taniguchi]{Kenji Taniguchi}
\address{
Department of Physics and Mathematics\\
Aoyama Gakuin University\\
Fuchinobe 5-10-1, Chuo-ku, Sagamihara 252-5258, Japan}
\email{taniken@gem.aoyama.ac.jp}

\subjclass[2000]{Primary 14M15; Secondary 14M17, 14M27}
\keywords{Symmetric pair, multi-flag variety.}

\theoremstyle{plain}
\newtheorem{theorem}{Theorem}

\newtheorem{corollary}[theorem]{Corollary}
\newtheorem{lemma}[theorem]{Lemma}

\newtheorem{itheorem}{Theorem}

\theoremstyle{definition}

\newtheorem{condition}[theorem]{Condition}
\newtheorem{problem}[theorem]{\upshape Problem}
\theoremstyle{remark}
\newtheorem{remark}[theorem]{\upshape Remark}

\numberwithin{equation}{section}
\numberwithin{theorem}{section}

\newcommand{\bbK}{\mathbb{K}}

\newcommand{\bbG}{\mathbb{G}}
\newcommand{\C}{\mathbb{C}}
\newcommand{\bbP}{\mathbb{P}}
\newcommand{\bbQ}{\mathbb{Q}}

\newcommand{\lie}[1]{\mathfrak{#1}}

\newcounter{thmenum}
\newenvironment{thmenumerate}{%
\begin{list}{$(\thethmenum)$}{%
\usecounter{thmenum}
\setlength{\labelsep}{.5em}
\setlength{\labelwidth}{-7pt}
\setlength{\topsep}{0pt}
\setlength{\partopsep}{0pt}
\setlength{\parsep}{0pt}
\setlength{\leftmargin}{3pt}
\setlength{\rightmargin}{0pt}
\setlength{\itemindent}{\leftmargin}
\setlength{\itemsep}{0pt}
}}
{\end{list}}

\newcommand{\mycomment}[1]{} 

\newlength{\lengthcup}
\newcommand{\rank}{\qopname\relax o{rank}}

\newcommand{\Ad}{\qopname\relax o{Ad}}
\newcommand{\ad}{\mathop{\mathrm{ad}}\nolimits{}}

\newcommand{\orbit}{\mathbb{O}}
\newcommand{\calorbit}{\mathcal{O}}

\newcommand{\GFl}{\mathfrak{X}}

\newcommand{\KFl}{\mathcal{Z}}
\newcommand{\GKFl}[2]{\GFl_{#1}\times\KFl_{#2}}

\newcommand{\clko}[1]{\mathrm{Cl}_K({#1})}

\begin{document}

\begin{abstract} 

Let $ G $ be a connected reductive algebraic group over $ \C $. 
We denote by $ K = (G^{\theta})_{0} $ the identity component of 
the fixed points of an involutive automorphism $ \theta $ of $ G $. 
The pair $ (G, K) $ is called a symmetric pair.   

Let $Q$ be a parabolic subgroup of $K$. 
We want to find a pair of parabolic subgroups $P_{1}$, $P_{2}$ of $G$ such that 
(i) $P_{1} \cap P_{2}  = Q$ and 
(ii) $P_{1} P_{2}$ is dense in $G$. 
The main result of this article states that, for a simple group $G$, 
we can find such a pair if and only if $(G, K)$ is a Hermitian symmetric pair. 

The conditions (i) and (ii) yield to conclude that the $K$-orbit through 
the origin $(e P_{1}, e P_{2})$ of $G/P_{1} \times G/P_{2}$ is 
closed and it generates an open dense $G$-orbit on the product of partial flag variety. 
From this point of view, we also give a complete classification of closed $K$-orbits on 
$G/P_{1} \times G/P_{2}$. 

\end{abstract}

\maketitle

\section{Review on double flag varieties for $ G/K $}

Let $ G $ be a connected reductive algebraic group over the complex number field 
$ \C $, and $ \theta $ its (non-trivial) involutive automorphism.  
The subgroup whose elements are fixed by $ \theta $ is denoted by $G^{\theta}$. 
We put $K=(G^{\theta})_{0}$, the identity component of $G^{\theta}$, 
and call it a symmetric subgroup of $ G $.  
We denote the Lie algebra of $ G $ (respectively of $ K $) by $ \lie{g} $ (respectively $ \lie{k} $).  
In the following, we use the similar notation; 
for an algebraic group we use a Roman capital letter, and for its Lie algebra the corresponding German small letter.

For a parabolic subgroup $ P $ of $G$, we denote a partial flag variety 
consisting of all $ G $-conjugates of $ P $ by $ \GFl_P $.  
We also choose a $ \theta $-stable parabolic $ P' $ in $ G $, and put $ Q = K \cap P' $.  
Then $ Q $ is a parabolic subgroup of $ K $, and every parabolic subgroup of $ K $ can be obtained in this way.
We denote a partial flag variety $ K/Q $ by $ \KFl_Q $.  
The product $ \GKFl{P}{Q} $ is called a \emph{double flag variety for the symmetric pair} $ (G, K) $.  
If there are only finitely many $ K $-orbits on the product $ \GKFl{P}{Q} $, 
it is called of \emph{finite type}.


Let us choose three parabolic subgroups $ P_1, P_2 $ and $ P_3 $ of $ G $.  
If one considers $ \bbG = G \times G $ and an involution $ \theta(g_1, g_2) = (g_2, g_1) $ of $ \bbG $, 
the symmetric subgroup $ \bbK = (\bbG^{\theta})_{0} $ is just the diagonal subgroup 
$ \Delta(G) \subset \bbG $.  
Thus $ (\bbG, \bbK) $ is a symmetric pair.  
Then $ \bbP = (P_1, P_2) $ is a parabolic subgroup of $ \bbG $ and $ \bbQ = \Delta(P_3) $ a parabolic subgroup of $ \bbK $, and 
our double flag variety can be interpreted as 
\begin{equation*}
\bbG/ \bbP \times \bbK/ \bbQ
= (G \times G)/(P_1 \times P_2) \times \Delta(G)/\Delta(P_3)
\simeq \GFl_{P_1} \times \GFl_{P_2} \times \GFl_{P_3}
\end{equation*}
which is nothing but the triple flag variety.  
So our double flag variety is a natural generalization of triple flag varieties.
The triple flag variety $ \GFl_{P_1} \times \GFl_{P_2} \times \GFl_{P_3} $ is said to be of finite type if there are 
finitely many $ G $-orbits in it.  

Let us return to the double flag variety $ \GFl_{P} \times \KFl_{Q} $.
One of the interesting problems is to classify the double flag varieties of finite type.  
In \cite{NO.2011}, Nishiyama and Ochiai gave two efficient criterions for the finiteness of orbits using triple flag varieties.  
Both criterions reduce the finiteness of orbits to that for a certain triple flag varieties.
The first one is

\begin{itheorem}[{\cite[Theorem~3.1]{NO.2011}}]
\label{NO.theorem:triple.flag.to.double.flag}
Let $ P' $ be a $ \theta $-stable parabolic of $ G $ such that $ P' \cap K = Q $.  
If the number of $ G $-orbits on $ \GFl_{P} \times \GFl_{\theta(P)} \times \GFl_{P'} $ is finite, 
then there are only finitely many $ K $-orbits on the double flag variety $ \GKFl{P}{Q} $.
\end{itheorem}

Here is the second one.

\begin{itheorem}[{\cite[Theorem~3.4]{NO.2011}}]
\label{NO.theorem:intersection.of.parabolics}
Let $ P_i \; (i = 1, 2, 3) $ be a parabolic subgroup of $ G $.  
Suppose that $ \GFl_{P_1} \times \GFl_{P_2} \times \GFl_{P_3} $
has finitely many $G$-orbits
and that $Q:= P_1 \cap P_2$ is a parabolic subgroup of $K$.
Then $\GKFl{P_3}{Q}$ has finitely many $K$-orbits.  

Moreover, if $ P_3 $ is a Borel subgroup $ B $ and the product $ P_1 P_2 $ is open in $ G $, 
then the converse is also true, i.e., 
the double flag variety 
$ \GKFl{B}{Q} $ 
is of finite type 
if and only if 
the triple flag variety 
$ \GFl_{P_1} \times \GFl_{P_2} \times \GFl_{B} $ 
is of finite type.
\end{itheorem}

The first criterion is a theoretical one, and the second one is easier to handle, though they overlap largely.  
In this paper, we are mainly interested in the second criterion and
its variant. 

The first main result of this article states that the condition in 
Theorem~\ref{NO.theorem:intersection.of.parabolics} is satisfied only
if $(G, K)$ is Hermitian. 
More precisely, if there exists a pair $(P_{1}, P_{2})$ of parabolic
subgroups of $G$ 
such that $Q = P_{1} \cap P_{2}$ is a parabolic subgroup of $K$ and 
$P_{1} P_{2}$ is open in $G$, then $(G, K)$ must be Hermitian 
(Theorem~\ref{theorem:main,non-Hermitian}). 
So we may restrict our interest to the Hermitian case. 
In this case, such pair $(P_{1}, P_{2})$ exits for any parabolic
subgroup $Q$ of $K$, 
and the classification of such pairs is obtained 
(Theorem~\ref{theorem:main,Hermitian}). 

Finding out such $Q, P_1, P_2$ is almost equivalent
to finding a closed $ K $-orbit inside the open $ G $-orbit in 
$\GFl_{P_1} \times \GFl_{P_2}$. 
In \S3, we give a classification of closed $ K $-orbits on 
the double flag variety $ \GFl_{P_1} \times \GFl_{P_2} $ 
(Theorem~\ref{thoerem:classification.of.closed.orbits}). 

\subsection*{Acknowledgment} 
For the proof of Lemma~\ref{lemma:existence.Q.implies.equal.rank}, 
discussion with Tohru Uzawa was very helpful.  
After we proved Theorem~\ref{theorem:main,non-Hermitian}, 
Hiroshi Yamashita suggested another concise proof and we followed it largely. 
We thank them very much.

\section{Intersection of parabolic subgroups}

Let $ Q $ be a parabolic subgroup of $ K $.  
Let us consider the following condition on $ Q $.

\begin{condition}
\label{condition:Q.is.PcapP.and.PP.dense}
There exists a pair of parabolic subgroups $ P_1, P_2 $ of $ G $ such that 
$ Q = P_1 \cap P_2 $ and the product $ P_1 \cdot P_2 $ is dense in $ G $.
\end{condition}

This condition is exactly the assumption 
of the latter half of Theorem~\ref{NO.theorem:intersection.of.parabolics}.  
Thus, 
under Condition~\ref{condition:Q.is.PcapP.and.PP.dense}, 
$ \GKFl{B}{Q} $ is of finite type if and only if 
$ \GFl_{P_1} \times \GFl_{P_2} \times \GFl_{B} $ is so.

Let us consider the following problem.

\begin{problem}
Let $ Q $ be a parabolic subgroup of $ K $.
\begin{thmenumerate}
\item
Classify all $ Q $ which satisfies Condition~\ref{condition:Q.is.PcapP.and.PP.dense} for a certain pair 
$ P_1, P_2 $ of parabolic subgroups of $ G $.
\item
If $ Q $ satisfies Condition~\ref{condition:Q.is.PcapP.and.PP.dense}, classify pairs $ P_1, P_2 $ up to $ K $-conjugate.
\end{thmenumerate}
\end{problem}

The first easy observation is the following.

\begin{lemma}
\label{lemma:existence.Q.implies.equal.rank}
If there is a parabolic subgroup $ Q $ of $ K $ which satisfies Condition~\ref{condition:Q.is.PcapP.and.PP.dense}, 
then $ \rank G = \rank K $ holds.  
In this case, parabolic subgroups $ P_1, P_2 $ are $ \theta $-stable.
\end{lemma}

\begin{proof}
Let $ P_1, P_2 $ be as in Condition~\ref{condition:Q.is.PcapP.and.PP.dense}.  
Since $ P_i $ is parabolic, it contains a Borel subgroup $ B_i \subset G $.  
For arbitrary chosen Borel subgroups $ B_1 $ and $ B_2 $, 
the intersection $ B_1 \cap B_2 $ contains a maximal torus $ T $ of 
$G$ (\cite[\S16 Exercise 8]{Humphreys.1972}).  
We have 
\begin{equation*}
T \subset B_1 \cap B_2 \subset P_1 \cap P_2 = Q \subset K ,
\end{equation*}
hence $ T $ is also a maximal torus of $ K $, which proves that $ \rank G = \rank K $.  
Now, the Lie algebra of $ P_1 $ admits a root space decomposition with respect to $ T $, 
hence it is $ \theta $-stable.
\end{proof}

Let $ G $ be a simple group.  
We say a symmetric pair $ (G, K) $ is of Hermitian type if 
the center of $ K $ is of positive dimension, and non-Hermitian otherwise.  
It is well known that, if the center of $ K $ has positive dimension, 
it must be one.  
Also, if $ (G,K) $ is Hermitian, then $ \rank G = \rank K $ holds, 
but the converse is not true.
Let $\lie{g} = \lie{k} \oplus \lie{s}$ be the Cartan decomposition of
$\lie{g}$ defined by (the differential of) $\theta$. 
It is also well known that $(G, K)$ is Hermitian if and only if the
adjoint representation $(\Ad, \lie{s})$ of $K$ on $\lie{s}$ is
reducible. 
Since we have assumed that $K (= (G^{\theta})_{0})$ is connected, 
the representation $(\Ad, \lie{s})$ of $K$ is reducible if and only if 
the adjoint representation $(\ad, \lie{s})$ of $\lie{k}$ is reducible.

\begin{theorem}\label{theorem:main,non-Hermitian}
Let $ G $ be a simple connected algebraic group. 
If there is a parabolic subgroup $ Q $ of $ K $ which satisfies 
Condition~\ref{condition:Q.is.PcapP.and.PP.dense}, 
then $(G,K)$ is of Hermitian type. 
\end{theorem}
\begin{proof}
Assume that there are parabolic subgroups $P_{1}$, $P_{2}$
which satisfy Condition~\ref{condition:Q.is.PcapP.and.PP.dense}. 
Since (i) $\lie{p}_{1}$, $\lie{p}_{2}$ are $\theta$-stable 
by Lemma~\ref{lemma:existence.Q.implies.equal.rank} and 
(ii) $\lie{g} = \lie{p}_{1} + \lie{p}_{2}$, 
$\lie{p}_{1} \cap \lie{p}_{2} \subset \lie{k}$ 
by Condition~\ref{condition:Q.is.PcapP.and.PP.dense}, 
the space $\lie{s}$ is a direct sum of subspaces 
$\lie{s} \cap \lie{p}_{1}$, $\lie{s} \cap \lie{p}_{2}$. 
These subspaces are non-zero. 
Actually, if $\lie{p}_{i} \cap \lie{s} = \{0\}$ for $i=1$ or $2$, 
then $\lie{p}_{j} \cap \lie{s} = \lie{s}$ for $j\not=i$. 
Since $\lie{g}$ is simple, $[\lie{s}, \lie{s}] = \lie{k}$. 
It follows that $\lie{p}_{j} \supset \lie{k} + \lie{s} = \lie{g}$, 
so 
$\lie{p}_{i} 
= \lie{p}_{i} \cap \lie{p}_{j} = \lie{q} \subset \lie{k}$. 
But this is impossible since $\lie{p}_{i}$ is a parabolic subalgebra
and $\lie{k}$ is a symmetric subalgebra. 

Since we have assumed that $\lie{p}_{1} \cap \lie{p}_{2}$ is a
parabolic subalgebra of $\lie{k}$, 
we can choose a Borel subalgebra $\lie{b}_{K}$ of $\lie{k}$ so that it
is contained in $\lie{p}_{1} \cap \lie{p}_{2}$. 
Then 
$\lie{s} 
= (\lie{s} \cap \lie{p}_{1}) \oplus (\lie{s} \cap \lie{p}_{2})$ is a
decomposition of the $\lie{b}_{K}$-module $(\mathrm{ad}, \lie{s})$. 
By the highest weight theory, 
$(\mathrm{ad}, \lie{s})$ is a reducible $\lie{k}$-module 
if and only if 
$(\mathrm{ad}, \lie{s})$ is a decomposable $\lie{b}_{K}$-module. 
We know that both $\lie{p}_{1} \cap \lie{s}$ and 
$\lie{p}_{2} \cap \lie{s}$ are non-zero. 
Therefore, $(\mathrm{ad}, \lie{s})$ is a reducible $\lie{k}$-module. 
It follows that $(G,K)$ is Hermitian. 
\end{proof}

\begin{remark}
Originally, we proved Theorem~\ref{theorem:main,non-Hermitian} 
by using the classification of simple symmetric pairs. 
Namely, we checked one by one that no parabolic subgroup $Q$
of $K$ satisfies Condition~\ref{condition:Q.is.PcapP.and.PP.dense} if
$(G,K)$ is non-Hermitian. 
Later, Hiroshi Yamashita suggested the above simpler proof to us, and we followed his suggestion.  
It much improves the proof of the theorem and we thank for his generous allowance to quote it.  
\end{remark}

By this theorem, we may restrict our interest to the Hermitian case. 

\begin{theorem}\label{theorem:main,Hermitian}
Let $ G $ be a simple connected algebraic group.
Assume that the pair $ (G, K) $ is of Hermitian type. 
Suppose $Q$ is any parabolic subgroup of $ K $. 
Let $\lie{s} = \lie{s}_{+} \oplus \lie{s}_{-}$ be the irreducible
decomposition of the adjoint representation of $K$ on $\lie{s}$. 
\begin{enumerate}
\item
The product 
$Q \exp \lie{s}_{\pm}$ is a parabolic subgroup of $G$. 
Let 
\begin{align}
&
P_{1} = Q \exp \lie{s}_{+}
&
&
P_{2} = K \exp \lie{s}_{-}.
\label{eq:a pair satisfying 2.1, type1}
\end{align}
%
%
Then the pair $(P_{1}, P_{2})$ satisfies 
Condition~\ref{condition:Q.is.PcapP.and.PP.dense}. 
\item
Suppose 
$(\lie{g}, \lie{k}) 
\simeq (\lie{sl}_{p+q}, \lie{sl}_{p} \oplus \lie{sl}_{q} \oplus \C)$ 
($p,q \geq 2$).  
By the classification of Hermitian symmetric pairs 
(see \cite{Knapp.2002} for example), 
this is the only case when $K$ is not simple modulo its center. 
Define 
\begin{align*}
& 
\lie{k}^{\mathrm{I}} := \lie{sl}_{p}
\supset 
\lie{q}^{\mathrm{I}} := \lie{k}^{\mathrm{I}} \cap \lie{q} 
\quad \mbox{and} \quad 
\lie{k}^{\mathrm{II}} := \lie{sl}_{q}
\supset 
\lie{q}^{\mathrm{II}} := \lie{k}^{\mathrm{II}} \cap \lie{q}. 
\end{align*}
Let $P_{1}$ and $P_{2}$ be the closed subgroups of $G$ whose Lie
algebras are 
\begin{align}
&
\lie{p}_{1} 
= 
(\lie{q}^{\mathrm{I}} \oplus \lie{k}^{\mathrm{II}} \oplus \C)
\oplus \lie{s}_{+}
\quad \mbox{and} \quad 
\lie{p}_{2} 
= 
(\lie{k}^{\mathrm{I}} \oplus \lie{q}^{\mathrm{II}} \oplus \C)
\oplus \lie{s}_{-}, 
\label{eq:a pair satisfying 2.1, type2}
\end{align}
respectively. 
%
%
Then $P_{1}$ and $P_{2}$ are parabolic subgroups of $G$, 
and the pair $(P_{1}, P_{2})$ satisfies
Condition~\ref{condition:Q.is.PcapP.and.PP.dense}. 
\item
Up to the exchange of the simple factors and/or the exchange of 
$\lie{s}_{+}$, $\lie{s}_{-}$, 
the cases \eqref{eq:a pair satisfying 2.1, type1} and 
\eqref{eq:a pair satisfying 2.1, type2} 
classify 
the pairs $(P_{1}, P_{2})$ which satisfy 
Condition~\ref{condition:Q.is.PcapP.and.PP.dense} for $Q$. 
\end{enumerate}
\end{theorem}
Note that, since $P_{i}$ ($i=1,2$) is connected, it is uniquely
determined by $\lie{p}_{i}$. 

\begin{proof}
(1), (2). 
Let $\lie{b}_{K}'$ be any Borel subalgebra of $\lie{k}$. 
Since $(G,K)$ is Hermitian, $\lie{b}_{K}' \oplus \lie{s}_{\pm}$ is
a Borel subalgebra of $\lie{g}$. 
It follows that all the groups appearing in (1) or (2) are parabolic
subgroups of $G$, since $K$ is connected. 
It is clear that the pairs $(P_{1}, P_{2})$ in (1) and (2) satisfy 
Condition~\ref{condition:Q.is.PcapP.and.PP.dense}. 

(3) For a parabolic subgroup $Q$ of $K$, assume that there exist
parabolic subgroups $P_{1}$, $P_{2}$ of $G$ which satisfy 
Condition~\ref{condition:Q.is.PcapP.and.PP.dense}. 
For the proof of (3), 
we will show that every simple factor of $K$ is contained in either
$P_{1}$ or $P_{2}$.

As in the proof of Theorem~\ref{theorem:main,non-Hermitian}, 
let $\lie{b}_{K}$ be a Borel subalgebra of $\lie{k}$ contained in 
$\lie{p}_{1} \cap \lie{p}_{2}$. 
Then both $\lie{s}_{+}$ and $\lie{s}_{-}$ are indecomposable
$\lie{b}_{K}$ modules by the highest weight theory. 
As a consequence of the proof of 
Theorem~\ref{theorem:main,non-Hermitian}, 
we may assume $\lie{s} \cap \lie{p}_{1}=\lie{s}_{+}$ and 
$\lie{s} \cap \lie{p}_{2}=\lie{s}_{-}$, 
by changing $\lie{s}_{+}$ and $\lie{s}_{-}$ if needed. 

Let $K_{s}$ be the connected subgroup of $K$ whose Lie algebra
$\lie{k}_{s}$ is a simple ideal of $\lie{k}$. 
The Borel subalgebra $\lie{b}_{K}$ defines a positive root system of
$\lie{k}_{s}$. 
Let $\gamma$ be the corresponding lowest root of $\lie{k}_{s}$, and
denote by $(\lie{k}_{s})_{\gamma}$ the lowest root space. 
By the proof of Theorem~\ref{theorem:main,non-Hermitian}, 
we have 
$\lie{p}_{1} + \lie{p}_{2} = \lie{g}$, so 
$(\lie{p}_{1} \cap \lie{k}_{s}) + (\lie{p}_{2} \cap \lie{k}_{s}) 
= \lie{k}_{s}$, 
since $\lie{p}_{1}$ and $\lie{p}_{2}$ are $\theta$-stable.  
Therefore, at least one of $\lie{p}_{1}$ and $\lie{p}_{2}$, 
say $\lie{p}_{2}$, 
contains the lowest root space $(\lie{k}_{s})_{\gamma}$. 
Since (i) $\lie{p}_{2}$ contains both $\lie{b}_{K} \cap \lie{k}_{s}$ and 
$(\lie{k}_{s})_{\gamma}$, and 
(ii) $\lie{b}_{K} \cap \lie{k}_{s}$ and $(\lie{k}_{s})_{\gamma}$
generate $\lie{k}_{s}$, 
the Lie algebra $\lie{k}_{s}$ is a subalgebra of $\lie{p}_{2}$. 
Since $K_{s}$ is connected, $K_{s}$ is contained in $P_{2}$, 
so $K_{s} \cap P_{2} = K_{s}$. 
In this case, 
$K_{s} \cap P_{1} = K_{s} \cap P_{1} \cap P_{2} = K_{s} \cap Q$.

Suppose $K$ is simple modulo its center. 
We have proved the followings: 
After exchanging $P_{1}$, $P_{2}$ and/or $\lie{s}_{+}$, $\lie{s}_{-}$
if needed, 
$P_{1}$ and $P_{2}$ satisfies 
$\lie{s} \cap \lie{p}_{1} = \lie{s}_{+}$, 
$\lie{s} \cap \lie{p}_{2} = \lie{s}_{-}$, 
$K \cap P_{1} = Q$ and $K \cap P_{2} = K$. 
Here, we used the fact that the center of $K$ is contained 
in $Q = P_{1} \cap P_{2}$. 
In this case, 
$P_{1} = Q \exp \lie{s}_{+}$ and $P_{2} = K \exp \lie{s}_{-}$. 
These are the groups in \eqref{eq:a pair satisfying 2.1, type1}. 
Just in the same way, we can show the case when $K$ is not simple
modulo its center. 
\end{proof}

\begin{remark}
If a symmetric pair $(G, K)$ is Hermitian, 
then the dimension of the center of $\lie{k}$ is one. 
But the converse is not always true if $K$ is not connected. 
For example, 
$(G, G^{\theta}) 
= (SO(n+2,\C), S(O(n,\C) \times O(2,\C)))$ is not Hermitian. 
Actually, the center of 
$G^{\theta} = S(O(n,\C) \times O(2,\C))$ is a finite group. 
On the other hand, 
$(G, K) 
= (SO(n+2,\C), SO(n,\C) \times SO(2,\C))$ is Hermitian. 
\end{remark}

\section{Closed orbits on double flag variety}

Let $ P_1 $ and $ P_2 $ be parabolic subgroups of $ G $.  
If $ Q' = K \cap P_1 \cap P_2 $ is a parabolic subgroup of $ K $, we have a natural embedding 
\begin{equation*}
K/Q' \hookrightarrow G/P_1 \times G/P_2 , \quad
k Q' \mapsto (k P_1, k P_2) .
\end{equation*}
Since $ K/Q' $ is a flag variety, it is compact, and the above embedding is a closed embedding.  
Thus we have a closed $ K $-orbit on $ \GFl_{P_1} \times \GFl_{P_2} $ which is isomorphic to $ K/Q' $.  

In particular, if $ Q = P_1 \cap P_2 $ is a parabolic subgroup of $ K $, 
$ K/Q $ is isomorphic to a closed $ K $-orbit on $ \GFl_{P_1} \times \GFl_{P_2} $.  
If, in addition to that, $ P_1 P_2 $ is dense in $ G $, 
the closed $ K $-orbit $ K \cdot (e P_1, e P_2) $ is in the open dense $ G $ orbit.  
Thus to find out $ Q, P_1, P_2 $ which satisfies Condition~\ref{condition:Q.is.PcapP.and.PP.dense} is 
almost equivalent to 
finding a closed $ K $-orbit inside the open $ G $-orbit in $ \GFl_{P_1} \times \GFl_{P_2} $.  

For this purpose, we will give a classification of closed $ K $-orbits on 
the double flag variety $ \GFl_{P_1} \times \GFl_{P_2} $ in terms of Weyl groups.

Let $ \calorbit \subset \GFl_{P_1} \times \GFl_{P_2} $ be a closed $ K $-orbit.  
For $ i = 1, 2 $, we denote the projection to the $ i $-th factor by 
$ \pi_i : \GFl_{P_1} \times \GFl_{P_2} \to \GFl_{P_i} $.  
Then $ \pi_i $ is a $ K $-equivariant map, and it brings $ K $-orbits to $ K $-orbits.  
Since $ \calorbit $ is compact by assumption, the image $ \orbit_i := \pi_i(\calorbit) $ is also compact, 
hence a closed $ K $-orbit on $ \GFl_{P_i} $.  
Let us denote the set of closed $ K $-orbits on a $ K $-variety $ \mathfrak{X} $ by 
$ \clko{\mathfrak{X}} $.  
Then the above correspondence gives a map 
\begin{align*}
\pi_{12} = \pi_1 \times \pi_2 & : 
\clko{\GFl_{P_1} \times \GFl_{P_2}} \to \clko{\GFl_{P_1}} \times \clko{\GFl_{P_2}} , \\
&
\pi_{12}(\calorbit) = (\orbit_1, \orbit_2).
\end{align*}

\begin{theorem}
\label{thoerem:classification.of.closed.orbits}
The map 
$ \pi_{12} : \clko{\GFl_{P_1} \times \GFl_{P_2}} \to \clko{\GFl_{P_1}} \times \clko{\GFl_{P_2}} $ 
above is bijective.  
In particular, there are finitely many closed $ K $-orbits on 
$ \GFl_{P_1} \times \GFl_{P_2} $.
\end{theorem}

\begin{proof}
To see that the map $ \pi_{12} $ is surjective, 
take closed orbits $ \orbit_i \in \clko{\GFl_{P_i}} \; (i = 1, 2) $.  
Since $ \pi_i^{-1}(\orbit_i) $ is a closed set, 
$ \pi_1^{-1}(\orbit_1) \cap \pi_2^{-1}(\orbit_2) $ is closed, 
hence contains a closed $ K $-orbit.  

Now we want to prove $ \pi_{12} $ is injective.  
So let us take a closed $ K $-orbit $ \calorbit \in \clko{\GFl_{P_1} \times \GFl_{P_2}} $.  
If $ (P_1', P_2') \in \calorbit $, 
then $ \orbit_i = K \cdot P_i' $.  
Put $ Q_i = P_i' \cap K $.  
Since $ K \cap P_1' \cap P_2' $ must be parabolic in $ K $, 
$ Q_1 \cap Q_2 $ is a parabolic subgroup.  
Since there is a unique closed $ K $ orbit in $ K/Q_1 \times K/Q_2 $ by Bruhat decomposition, 
the choice of $ (Q_1, Q_2) $ is unique up to diagonal $ K $-conjugate.   
Thus the possibility of $ (P_1', P_2') $ is also unique up to diagonal $ K $-action.
\end{proof}

We can determine the number of closed orbits using 
the classification of closed $ K $-orbits on $ \GFl_P $ by \cite{Brion.Helminck.2000}.  
To quote it, we need notation.  

Let $ B \subset G $ be a $ \theta $-stable Borel subgroup and 
take a $ \theta $-stable maximal torus $ T $ in $ B $.  
We consider root system $ \Delta = \Delta(\lie{g}, \lie{t}) $, Weyl group $ W_G = N_G(T)/Z_G(T) $ etc. 
with respect to this $ T $, 
and choose a positive system $ \Delta^+ $ corresponding to $ B $.  
Then $ \Delta^+ $ determines a simple system $ \Pi $.  
Since $ B $ and $ T $ are $ \theta $-stable, 
$ \theta $ naturally acts on $ W_G $ and $ \Delta $, 
and preserves $ \Delta^+ $ and $ \Pi $.  
Let $ W_G^{\theta} $ be a subgroup of $ W_G $ whose elements are fixed by $ \theta $.  

Since $ B_K = K \cap B $ is a Borel subgroup of $ K $, 
it contains a maximal torus $ T_K $ of $ K $.  We may assume that $ T_K = T^{\theta} $.  
Then $ W_K = N_K(T_K) / Z_K(T_K) $ can be identified with a subgroup of $ W_G $ 
(see \cite[p.~280]{Brion.Helminck.2000}, for example).  

We consider standard parabolic subgroups containing $ B $.  
If $ P $ is a standard parabolic subgroup of $ G $, 
then $ P $ is determined by a subset $ J $ in $ \Pi $; 
the root subsystem $ \Delta_J $ generated by $ J $ is the root system of a Levi component $ L $ of $ P $.  
We always take $ L $ as an algebraic subgroup 
whose Lie algebra is the sum of root subspaces of $ \Delta_J $ and $ \lie{t} $.  
This correspondence is a bijection between the standard parabolic subgroups of $ G $ and the subsets of $ \Pi $.  
If $ P $ corresponds to $ J $, sometimes we will write $ P = P_J $.  
Then $ \theta $-stable parabolic subgroups correspond exactly to the $ \theta $-stable subsets in $ \Pi $.  
Also we denote the Weyl group of $ \Delta_J $ by $ W_J $ or $ W_P $.  
$ W_P^{\theta} $ denotes the subgroup of $ W_P $ whose elements are fixed by $ \theta $.  
(Though $ \theta $ does not preserve $ P $ always, $ W_P^{\theta} $ makes sense.)

\begin{theorem}[{\cite[Proposition~9]{Brion.Helminck.2000}}]
The set of closed $ K $-orbits on $ \GFl_P $ corresponds bijectively to 
$ W_K \backslash W_G^{\theta} / W_P^{\theta} $.  
Bijection simply maps $ W_K \dot{w} W_P^{\theta} $ to $ K \dot{w} P $, 
where $ \dot{w} \in N_G(T) $ is a representative 
of an element of $ W_G^{\theta} $. 
\end{theorem}

Two remarks are in order.

First, if $ P $ is not $ \theta $-stable, 
let $ P' = P \cap \theta(P) $ be the largest $ \theta $-stable parabolic contained in $ P $.  
Then closed $ K $-orbits on $ \GFl_P $ and those on $ \GFl_{P'} $ are in bijection.

Second, 
if $ \rank G = \rank K $, we can assume $ T_K = T $ above.  
Then, clearly $ \theta $ acts on $ W_G $ as an identity.  Thus we get 
$ W_K \backslash W_G^{\theta} / W_P^{\theta} = W_K \backslash W_G / W_P $.  

We can deduce the number of closed orbits on $ \GFl_{P_1} \times \GFl_{P_2} $ immediately.

\begin{corollary}
The number of closed $ K $ orbits on $ \GFl_{P_1} \times \GFl_{P_2} $ 
is equal to $ \# W_K \backslash W_G^{\theta} / W_{P_1}^{\theta} \times 
\# W_K \backslash W_G^{\theta} / W_{P_2}^{\theta} $.  
If $ \rank G = \rank K $, it reduces to 
$ \# W_K \backslash W_G / W_{P_1} \times \# W_K \backslash W_G / W_{P_2} $.  
\end{corollary}



\def\cftil#1{\ifmmode\setbox7\hbox{$\accent"5E#1$}\else
  \setbox7\hbox{\accent"5E#1}\penalty 10000\relax\fi\raise 1\ht7
  \hbox{\lower1.15ex\hbox to 1\wd7{\hss\accent"7E\hss}}\penalty 10000
  \hskip-1\wd7\penalty 10000\box7} \def\cprime{$'$}
\providecommand{\bysame}{\leavevmode\hbox to3em{\hrulefill}\thinspace}
\providecommand{\MR}{\relax\ifhmode\unskip\space\fi MR }
\providecommand{\MRhref}[2]{%
  \href{http://www.ams.org/mathscinet-getitem?mr=#1}{#2}
}
\providecommand{\href}[2]{#2}

\end{document}